\documentclass[12pt]{amsart}
\usepackage{amsfonts,amssymb,amscd,amsmath,enumerate,verbatim,calc}
\usepackage[colorlinks,linkcolor=blue,pagebackref=true, citecolor=red]{hyperref}
\usepackage{hyperref}
\usepackage[all]{xy}
\usepackage[mathscr]{euscript} 
\usepackage{ mathrsfs }
\usepackage[left=2.25cm,right=2.25cm,top=2.5cm,bottom=2.5cm]{geometry}

\usepackage{graphicx}
\usepackage{color}

\newcommand{\nn}{\underline{n} }
\newcommand{\mm}{\underline{m} }
\newcommand{\II}{\underline{I} }
\newcommand{\n}{\mathfrak{n} }
\newcommand{\m}{\mathfrak{m} }

\newcommand{\fp}{\mathfrak{p}}
\newcommand{\f}{\mathcal{F}}
\newcommand{\R}{\mathcal{R}}

\newcommand{\N}{\mathbb{N} }

\newcommand{\I}{\mathcal{I}}

\newcommand{\QQ}{\mathbb{Q} }

\newcommand{\C}{\mathbb{C} }

\newcommand{\V}{{\rm V}}

\newcommand{\Ass}{\operatorname{Ass}}

\newcommand{\depth}{\operatorname{depth}}
\newcommand{\grade}{\operatorname{grade}}

\newcommand{\gr}{\operatorname{gr}}
\newcommand{\Frac}{\operatorname{Frac}}
\newcommand{\sing}{\operatorname{Sing}}
\newcommand{\jac}{\operatorname{Jac}}

\newcommand{\ann}{\operatorname{ann}}

\newcommand{\hgt}{\operatorname{ht}}

\providecommand\Spec{\text{\rm Spec}}

\newcommand{\chr}{\operatorname{char}}

\theoremstyle{plain}

\newtheorem{theorem}{Theorem}[section]
\newtheorem{corollary}[theorem]{Corollary}
\newtheorem{lemma}[theorem]{Lemma}
\newtheorem{proposition}[theorem]{Proposition}

\theoremstyle{definition}
\newtheorem{definition}[theorem]{Definition}
\newtheorem{s}[theorem]{}
\newtheorem{remark}[theorem]{Remark}
\newtheorem{remarks}[theorem]{Remarks}
\newtheorem{example}[theorem]{Example}
\newtheorem*{example*}{\it Example}

\theoremstyle{remark}
\newtheorem*{claim*}{\it Claim}

\newtheorem*{case*}{\it Case}

\title[Eakin-Sathaye theorem]{Eakin-Sathaye type theorems \\ [2mm]  for joint reductions and good filtrations of ideals}
\thanks{{\it 2010 AMS Mathematics Subject Classification:} 13A30 }
\thanks{{\it Key words}: Eakin-Sathaye theorem, good filtrations, equimultiple good filtrations, joint reduction}
\thanks{Both of the first author and the second author are supported by UGC fellowship, Govt. of India}

\author{Kriti Goel}
\address{Department of Mathematics, Indian Institute of Technology Bombay, Mumbai, 400076, India}
\email{kriti@math.iitb.ac.in}
\author{Sudeshna Roy}
\address{Department of Mathematics, Indian Institute of Technology Bombay, Mumbai, 400076, India}
\email{sudeshnaroy@math.iitb.ac.in}
\author{J. K. Verma}
\address{Department of Mathematics, Indian Institute of Technology Bombay, Mumbai, 400076, India}
\email{jkv@math.iitb.ac.in}

\begin{document}

\begin{abstract}
Analogues of Eakin-Sathaye theorem for reductions of ideals are proved for $\N^s$-graded good filtrations. These analogues  yield bounds on joint reduction vectors for a family of ideals and reduction numbers for $\N$-graded filtrations.  Several examples related to lex-segment ideals, contracted ideals in $2$-dimensional regular local rings and the filtration of integral and tight closures  of powers of ideals in hypersurface rings are constructed to show  effectiveness of these bounds.
\end{abstract}

\maketitle
\centerline{\em Dedicated to Le Tuan Hoa on the occasion of his sixtieth birthday}
\section{Introduction}

The objective of this paper is to prove Eakin-Sathaye type theorems \cite{ES1976} for joint reductions and good filtrations of ideals. Recall that an ideal $J$ contained in an ideal $I$ in a commutative ring $R$ is called a reduction of $I$ if there is a non-negative integer $n$ such that $JI^n=I^{n+1}.$ The concept of reduction of an ideal was introduced by Northcott and Rees \cite{NR1954}. It has  become an important tool in many investigations in commutative algebra and algebraic geometry such as Hilbert-Samuel functions \cite{rees1961}, blow-up algebras \cite{GS1979}  singularities of  hypersurfaces \cite{teissier1973}, number of defining equations of algebraic varieties
\cite{lyubeznik1986} and many others.

Research in this paper is motivated by the following result of Paul Eakin and Avinash Sathaye \cite{ES1976}.  Let $\mu(I)$ denote the minimum number of generators of an ideal $I$ in a  local ring.

\begin{theorem} 
	Let $I$ an ideal of a local  ring $R$ with infinite residue field. If $\mu(I^n)< \binom{n+r}{r}$ for some positive integers $n$ and $r$, then there is a reduction $J$ of $I$ generated by $r$ elements such that $JI^{n-1}=I^n.$
\end{theorem}
The case of $n=2, r=1$ was proved by J. D. Sally \cite{sally1975}. The EST (Eakin-Sathaye Theorem) has been  revisited by G. Caviglia \cite{caviglia},  N. V. Trung \cite{trung} and Liam O'Carroll \cite{carroll}. Caviglia used Green's hyperplane section theorem to give a new proof of the EST. The EST was generalised by Liam O'Carroll for complete reductions
\cite{carroll}. The versions of the EST proved in this paper follow the approach used by O'Carroll. In order to state his result we recall necessary definitions and results about complete reductions and joint reductions of a family of ideals introduced by Rees in \cite{rees1984}. Recall that the {\em analytic spread } of an ideal $I$ in a local ring $(R,\m)$ 
is the Krull dimension of the  fiber cone of $I,$ namely, $F(I)=\oplus_{n=0}^\infty I^n/\m I^n .$ The analytic spread of $I$ is denoted by $\ell(I).$ Let $I_1, I_2,\dots, I_g$ be ideals of  $R$ with  $s=\ell(I_1I_2\cdots I_g).$ Then there is a $g\times s$ matrix $A=(a_{ij})$ with entries $a_{ij}\in I_i$ for $i=1,2,\dots, g$ and $j=1,2,\dots, s$ such that the ideal
$(y_1,y_2,\dots, y_s)$  is a reduction of the product $I_1I_2\cdots I_g.$ Here $y_j=\prod_{i=1}^ga_{ij}$ for $j=1,2,\dots,s$ and the set of elements $a_{ij}$, $i=1,\ldots,g,$ $j=1,\ldots,s$, is called a complete reduction of the set of ideals $I_1,\ldots,I_g.$ 

Let $\dim R=d$ and $I_1, I_2,\dots, I_d$ be $\m$-primary ideals of $R.$ Let $e_1,e_2, \dots, e_d$ be the standard basis of the $d$-dimensional $\QQ$-vector space $\QQ^d.$ If $\nn=(n_1,\dots, n_d) \in \N^d$ then we write $\underline{I}^{\nn}=I_1^{n_1}I_2^{n_2}\cdots I_d^{n_d}.$ A set of elements $(a_1, a_2,\dots, a_d)$ where $a_i\in I_i$  for $i=1,2,\dots, d$ is called a {\em joint reduction}  of the set of ideals $(I_1, I_2,\dots, I_d)$ if there exists $\underline{r}=(r_1, r_2,\dots, r_d)\in \N^d$ such that for all $n_i\geq r_i$ for $i=1,2,\dots, d,$
$$\sum_{i=1}^d a_i \underline{I}^{\nn-e_i}=\underline{I}^{\nn}.$$
The vector $(r_1, r_2,\dots, r_d)$ is called a {\em joint reduction vector} of $I$ with respect to the joint reduction $(a_1, a_2, \dots, a_d).$ Rees proved that if the set of elements $a_{ij}$, $i,j=1,\ldots,d$ is a complete reduction of $I_1,\ldots,I_d$ and if $i\rightarrow j(i)$ is a permutation of $1,\ldots,d$, then the set of elements $a_{ij(i)}$, $i=1,\ldots,d$ is a joint reduction of $I_1,\ldots,I_d.$ O'Carroll gave the following result for complete reductions in \cite{carroll}.

\begin{theorem}[\bf L. O'Carroll] 
	Let $(R,\m)$ be a Noetherian local ring with infinite residue field $k:=R/\m.$ Let $I_1, I_2,\dots, I_s$ be ideals in $R$ and $I=I_1I_2\cdots I_s.$
	Suppose that for some $n\geq 1$ and $r\geq 0$ $\mu(I^n) < \binom{n+r}{r}.$ Then there exist ``general"  elements $y_1, y_2,\dots, y_r$ with $y_j=x_{1j}\cdots x_{sj}, j=1,2,\dots, r$ where $x_{ij}\in I_i$ for $i=1,2,\dots, s$ such that
	$$(y_1, y_2, \dots, y_r)I^{n-1}=I^n.$$
\end{theorem}
We now describe the main results proved in this paper. An $\N^s$-graded filtration of ideals $\{I_{\nn}\}_{\nn\in \N^s}$ in $R$ is a collection of ideals which satisfies the conditions
(1) $I_{\nn}\subseteq I_{\mm}$ for all ${\nn} \geq \mm$  and (2) $I_{\nn}I_{\mm} \subseteq I_{\nn+\mm}$ 
for all $\nn, \mm \in \N^s.$ We say that the filtration  $\mathcal F =\{I_{\nn}\}_{\nn\in \N^s}$  is {\em good} if the Rees algebra $\mathcal R(\mathcal F)=\oplus_{\nn \in \N^s}I_{\nn} \underline{t}^{\nn}$ is a finite module over the Rees algebra $\mathcal R(I_{e_1},\dots, I_{e_s})=\oplus_{\nn \in \N^s} I_{e_1}^{n_1}\cdots I_{e_s}^{n_s}\underline{t}^{\nn}.$ Here $\underline{t}^{\nn}=t_1^{n_1}\dots t_s^{n_s}$ where $t_1, \dots, t_s$ are indeterminates and $\nn=(n_1, \dots, n_s)\in \N^s$ and $\mm=(m_1,\ldots,m_s)\geq\nn=(n_1,\ldots,n_s)$ if $m_i\geq n_i$ for all $i= 1, \ldots, s.$ We shall prove the following result in Section 3.

\begin{theorem}\label{1.3}
	Let $(R, \m, k)$ be a Noetherian local ring with infinite residue field $k$ and let $\f=\{I_{\underline{n}}\}_{\nn \in \N^s}$ be an $\N^s$-graded good filtration in $R.$ 
	Suppose 
	\[\mu(I_{\underline{n}})< \binom{n_1+r_1}{r_1} \cdots \binom{n_s+r_s}{r_s}\] 
	for some integers $n_1+\cdots+n_s \geq 1$ and $r_1+\cdots+r_s \geq 0$. Let $\underline{a}$ be the maximum of the degree of the generators of $F=F(\mathcal{F})$ as a module over $G=F(I_{e_1}, \ldots, I_{e_s}),$ where the maximum is taken component-wise. If $\underline{n} \geq \underline{a}+\underline{1}$, then for all $1 \leq i \leq s$ there exist ``general" elements $x_{i1}, \dots, x_{ir_i} \in I_{e_i}$ such that $I_{\underline{n}}= \sum_{i=1}^s (x_{i1}, \dots, x_{ir_i})I_{\underline{n}-e_i}.$
\end{theorem}

In particular, we get the following result for product of adic filtration.
\begin{corollary}\label{cor}
	Let $(R, \m, k)$ be a Noetherian local ring with infinite residue field $k$ and $I_1,\ldots, I_s$ be ideals in $R$. Suppose 
	\[\mu(I_1^{n_1}\cdots I_s^{n_s})< \binom{n_1+r_1}{r_1}\cdots \binom{n_s+r_s}{r_s}\] 
	for some integers $n_1+\cdots+n_s \geq 1$ and $r_1+\cdots+r_s \geq 0$. Then for all $1 \leq i \leq s$ there exist ``general" elements $x_{i1}, \dots, x_{ir_i} \in I_i$ such that for all $\underline{m} \geq \underline{n}$, 
	\[I_1^{m_1}\cdots I_s^{m_s}= \sum_{i=1}^s (x_{i1}, \dots, x_{ir_i})I_1^{m_1}\cdots I_i^{m_i-1}\cdots I_s^{m_s}.\]
\end{corollary}

Example \ref{counter} illustrates how Corollary \ref{cor} improves O'Carroll's generalisation of the EST for joint reductions. The motivation comes from the following observation. Let $I_1,I_2$ be ideals in a $2$-dimensional ring. Let $x_{11}, x_{12} \in I_1$ and $x_{21},x_{22} \in I_2$ such that under the hypothesis of Carroll's result, we get $(x_{11}x_{21},x_{12}x_{22})I_1^{n-1}I_2^{n-1} = I_1^n I_2^n.$ This gives a joint reduction equation of the form $x_{11}I_1^{n-1}I_2^n+x_{22}I_1^nI_2^{n-1}=I_1^n I_2^n.$ Whereas, the joint reduction equation obtained from Corollary \ref{cor}, in this case, is of the form $x_{11}I_1^{n_1-1}I_2^{n_2} + x_{22}I_1^{n_1} I_2^{n_2-1} = I_1^{n_1} I_2^{n_2},$ where $n_1$ may not be equal to $n_2.$

In Section 4, we consider an analogue of the EST to estimate the  reduction number of a good $\N$-graded filtration. For this  we need depth conditions on the associated graded ring and  the fiber cone.
In addition, we need the notion of equimultiple filtration: 

\begin{definition}
	Let $\f$ be an $\N$-graded good filtration. Define $l(\f) = \dim(F(\f))$ to be the analytic spread of the filtration $\f$. We say that the filtration $\f$ is \emph{equimultiple} if $l(\f) = \hgt I_1$.
\end{definition}

\begin{theorem}
	Let $(R, \m)$ be a Cohen-Macaulay  local ring with $R/\m=k$ infinite. Let $\f=\{I_n\}_{n \in \N}$ be an equimultiple good filtration such that $\grade\gr_{\f}(R)_+ \geq l(\f)=r$ and $F(\f)$ is Cohen-Macaulay. Let $\mu(I_n)< \binom{n+r}{r}$ for some $n\ge 1$. Then there exist $r$ general elements $x_1, \ldots, x_r \in I_1$ such that $I_m = (x_1,\ldots,x_r)I_{m-1}$ for all $m \geq n$. 
\end{theorem}

Finally in Section 5, we present several examples which illustrate our results and explain the necessity of depth assumption for the EST for reduction number of good filtrations. 

{\bf Acknowledgements:} We thank the referee for a very careful reading of the manuscript and suggesting several improvements.

\section{Preliminaries}

In this section, we setup notation, recall definitions and results which are required in the subsequent sections.

\s {\bf Multi-graded filtrations of ideals}

Let $(R,\m)$ be a $d$-dimensional Noetherian local ring and $I_1, \ldots, I_s$ be ideals in $R$. For $s\geq 1,$ we put $\underline{0}=(0,\ldots,0)\in{\N}^s$, $\underline{1}=(1,\ldots,1)\in{\N}^s$ and $e_i=(0,\ldots,1,\ldots,0)\in{\N}^s$ where $1$ occurs at the $i$-th position. Let $\nn=(n_1,\ldots,n_s)\in{\N}^s,$ then we write $\underline{I}^{\nn}=I_{1}^{n_1}\cdots I_{s}^{n_s}$ and we put $|\nn|=n_1+\cdots+n_s.$ By the phrase ``for all large $\nn$" we mean $n_i\gg 0$ for all $i= 1, \ldots, s$. 

\begin{definition}
	A set of ideals $\f=\{I_{\nn}\}_{\nn \in \N^s}$ is called an $\N^s$-{\it graded filtration} if for all $\mm,\nn\in\N^s, I_{\nn}I_{\mm}\subseteq I_{\nn+\mm}$ and if $\mm\geq\nn,$ $I_{\mm}\subseteq I_{\nn}$. Moreover, $\f$ is called an $\N^s$-graded {\it{$\underline{I}=(I_1, \ldots, I_s)$-filtration}} if $\underline{I}^{\nn}\subseteq I_{\nn}$ for all $\nn\in\N^s$.
\end{definition}

Let $t_1,\ldots,t_s$ be indeterminates. For $\nn\in\N^s,$ put $\underline t^{\nn}=t_{1}^{n_1}\cdots t_{s}^{n_s}$ and denote the $\N^s$-graded {\it{Rees ring of $\f$}} by $\mathcal{R}(\f)=\bigoplus\limits_{\nn \in \N^s}
I_{\nn}~{\underline{t}}^{\nn}$.
For $\f=\{\II^{\nn}\}_{\nn \in \N^s}$, we set $\mathcal R(\II)=\mathcal R(\f)$. 
The fiber cone of the filtration $\f$ is denoted by $F(\f)=\mathcal{R}(\f) \otimes_R R/\m= \bigoplus\limits_{\nn \in \N^s}
I_{\nn}/\m I_{\nn}$. Define $l(\f) = \dim F(\f)$ to be the {\it analytic spread} of the filtration $\f$. We say 
$F=\bigoplus_{\nn \in \N^s} F_{\nn}$ is {\it standard $\N^s$- graded algebra over $k$} if $F=k[F_{e_1}, \ldots, F_{e_s}]$.

\begin{definition}
	An $\N^s$-graded filtration $\f=\lbrace I_{\nn}\rbrace_{\nn\in \N^s}$ of ideals in $R$ is called an $\underline{I}=(I_1, \ldots, I_s)$-{\emph{good filtration}} if $I_i \subseteq I_{e_i}$ for all $i=1,\ldots,s$ and $\mathcal{R}(\f)$ is a finite $\mathcal{R}(\II)$-module.
\end{definition}
\noindent
If $R$ is an analytically unramified local ring and $I$ is an ideal of $R$, then Rees \cite{reesAU} proved that the integral closure filtration $\f=\{\overline{I^n}\}$ is an $I$-good filtration. Using \cite{GMV}, under the same conditions, the tight closure filtration $\mathcal{T} = \{(I^n)^*\}$ is an $I$-good filtration. 

A \emph{reduction} of a good filtration $\f = \{I_n\}$ is an ideal $J \subseteq I_1$ such that $JI_n = I_{n+1}$ for all large $n.$ Equivalently, $J \subseteq I_1$ is a reduction of $\f$ if and only if $\R(\f)$ is a finite $\R(J)$-module. A minimal reduction of $\f$ is a reduction of $\f$ minimal with respect to containment. Minimal reductions of a good filtration always exist and are generated by $l(I_1)$ elements if the residue field is infinite.

\begin{remarks}
	(1) Let $\mathcal{G}=\{J_n\}_{n \geq 0}$ be a $J$-good filtration. Then $J_{n+1}=JJ_n$ for all large $n.$ Since $J \subseteq J_1,$ it follows that $J_{n+1}=JJ_n \subseteq J_1 J_n \subseteq J_{n+1}$ and hence $J_{n+1}=J_1J_n$ for all large $n$. This shows that $\mathcal{G}$ is also a $J_1$-good filtration. 
	Some basic facts on good filtration are given in the paper \cite{hoaZarzuela}.

	(2) Let $\f=\{I_{\underline{n}}\}_{\underline{n}\in \N^s}$ be an $\N^s$-graded good filtration. Then $\mathcal{R}(\f)$ is a finite $\mathcal{R}(\II)$-module by definition, where $\II=(I_{e_1},\ldots,I_{e_s}).$ Set $G=F(I_{e_1}, \ldots, I_{e_s})= \bigoplus_{\underline{m} \in \N^s} I_{e_1}^{m_1}\cdots I_{e_s}^{m_s}/\m I_{e_1}^{m_1}\cdots I_{e_s}^{m_s}$, the fiber cone of the ideals $I_{e_1}, \ldots, I_{e_s}$ and $F=F(\f)= \bigoplus_{\underline{m} \in \N^s} I_{\underline{m}}/\m I_{\underline{m}}$ the fiber cone of the filtration $\f$. Note that $G$ is a standard multi-graded $k$-algebra and $F$ is a finitely generated $G$-module. Set $G^{(i)}= \bigoplus_{m \geq 0} G_{me_i}$ for all $1 \leq i \leq s$. Clearly $G^{(i)}$ is a standard graded $k$-algebra with $G^{(i)}_1=G_{e_i}$.

	(3) Let $\f=\{I_{\nn}\}_{\nn \in \N^s}$ be a good filtration. Then $\f_i=\{I_{ne_i}\}_{n \ge 0}$ is a good filtration for all $i=1,\ldots,s.$ For, it is sufficient to show that $\R(\f_i)$ is a finite $\R(I_{e_i})$-module. Consider the ideal 
	\[\I_{(i)} = \bigoplus_{\nn \in \N^s \setminus \N e_i} \ I_{\nn}~\underline{t}^{\nn} \ \subseteq \ \R(\f).\]
	Observe that 
	\[\I_{(i)} \cap \R(\II) =  \bigoplus_{\nn \in \N^s \setminus \N e_i} \ I_{e_1}^{n_1}\cdots I_{e_s}^{n_s}~\underline{t}^{\nn} .\] 
	As $\R(\f)$ is a finite $\R(\II)$-module, it follows that $\R(\f)/\I_{(i)}$ is a finite $\R(\II)/(\I_{(i)} \cap \R(\II))$-module. Since $\R(\f)/\I_{(i)} \simeq \R(\f_i)$ and $\R(\II)/(\I_{(i)} \cap \R(\II)) \simeq \R(I_{e_i})$, we are done.
\end{remarks}

\s {\bf Zariski topology}

Let $V$ be a finite dimensional $k$-vector space and $\dim_k V=N$. Then we can identify any vector $v \in V$ with an element of $k^N$. By a (non-empty) Zariski-open set in $V^r$, for $r \in \N$, we mean a finite union of sets of the form $X_f:=\{{\bf a}\in k^{Nr}\mid f({\bf a})\neq 0 \}$ for a given non-zero polynomial $f \in k[X_1, \ldots, X_{Nr}]$, see \cite[Exercise 17]{atiyahMacd}. Since for non-zero polynomials $f$ and $g$ in $k[X_1, \ldots, X_{Nr}]$, $X_f \cap X_g = X_{fg}$ so finite intersection of any two non-empty open set in $V^r$ is non-empty. In fact, intersection of finitely many non-empty open sets in $V^r$ is non-empty.

\begin{remark}\label{rmk2} 
	Let $U$ be a subspace of $V$ with $\dim_k U=m$ and $\dim_k V=n$. Note that $\dim_k V/U=n-m=t$ (say)  and $m, t \leq n$. We claim that the quotient map $\pi: V \to V/U$ is continuous in Zariski topology. Since $X_{f}$ is a basic open set of $V/U$ for some $f \in k[X_1, \ldots, X_t]$, it is enough to show that $\pi^{-1}(X_f)$ is a Zariski-open subset of $V$. Notice $\pi^{-1}(X_f) \simeq X_f \times k^m$ is a basic open subset of $V$ defined by $f \in k[X_1, \ldots, X_t]\subseteq k[X_1, \ldots, X_n]$. 
\end{remark}

\s {\bf General elements}

Let $F = \oplus_{n \geq 0} F_n$ be a standard graded algebra over a field $k$ and suppose that there exists a $k$-vector space epimorphism $\phi: V \to F_1$. Let $\mathcal{P}$ be a property of elements of $F_1$ and let $r \geq 0$. We say that {\it $\mathcal{P}$ holds for $r$ general elements} $y_1,\ldots, y_r$ of $F_1$ if there exists a non-empty Zariski-open subset $U$ of $V^r$ such that $\mathcal{P}$ holds for every sequence of elements $y_j := \phi(v_j )$ where $(v_1,\ldots, v_r )\in U$.

Let $(R,\m,k)$ be a Noetherian local ring and $I$ be an ideal in $R$. Set $F(I)=\bigoplus_{n \geq 0}I^n/\m I^n$, the fiber cone of $I$. We say $x$ is a ``general" element in $I$ if $x$ is a general element in $F(I)_1$, see \cite[Remark 3.2]{carroll}.

\section{Eakin-Sathaye theorem for multi-graded filtration}

In this section, we generalize the Eakin-Sathaye theorem for multi-graded good filtrations. We prove a lemma first. Set $F_{\underline{n}}=0$ when $\underline{n} \notin \mathbb{N}^s.$

\begin{lemma}\label{lem1}
	Let $G=\bigoplus_{\underline{m} \in \N^s} G_{\underline{m}}$ be a standard $\N^s$-graded algebra over a field $k$. Let $F= \bigoplus_{\underline{n} \in \N^s} F_{\underline{n}}$ be an $\N^s$-graded $k$-algebra and a finitely generated $G$-module such that $G_{\underline{0}}=F_{\underline{0}}=k$ and $G_{e_i}=F_{e_i}$ for all $i=1, \ldots, s$. 
	Let $y_{i1}, \ldots, y_{ip_i}\in G_{e_i}$ be a basis of $G_{e_i}$. Let $y=y_{ij}$ for some $1\leq i \leq s$ and $1 \leq j \leq p_i$. Set $\overline{G}=G/yG =\oplus_{\nn \in \N^s} \overline{G}_{\nn}$ and $\overline{F}=F/yF=\oplus_{\nn \in \N^s} \overline{F}_{\nn}$.  
	Let for some $r_i \geq 2$ there exist general elements $\overline{a_1}, \ldots, \overline{a_{r_i-1}} \in \overline{G}_{e_i}$ and $\overline{b_{t1}}, \ldots, \overline{b_{tr_t}} \in \overline{G}_{e_t}$ for all $t \neq i$ such that 
	\[	\overline{F}_{\nn}= (\overline{a_1}, \ldots, \overline{a_{r_i-1}}) \overline{F}_{\nn-e_i}+ \sum_{\substack{t=1 \\ t\neq i}}^{s} (\overline{b_{t1}}, \ldots, \overline{b_{tr_t}})\overline{F}_{\nn-e_t}.\]
	Then $y, a_1, \ldots, a_{r_i-1}$ are general elements in $G_{e_i}$ and $b_{t1}, \ldots, b_{tr_t}$ are general elements in $G_{e_t}$ for all $t \neq i$ such that \[F_{\nn}= (y, a_1, \ldots, a_{r_i-1}) F_{\nn-e_i}+ \sum_{\substack{t=1 \\ t\neq i}}^{s} (b_{t1}, \ldots, b_{tr_t})F_{\nn-e_t}.\] 
\end{lemma}

\begin{proof}
	Note that $\overline{G}= \bigoplus_{\nn \in \N^s}G_{\nn}/yG_{\nn-e_i}$ and $\overline{F}= \bigoplus_{\nn \in \N^s}F_{\nn}/yF_{\nn-e_i}$. Hence $\overline{G}_{e_i}=G_{e_i}/ky$ and $\overline{G}_{e_t}=G_{e_t}$ for all $t \neq i$.
	Let $V= V_1 \times \cdots \times V_s$ be a Cartesian product of finite-dimensional $k$-vector spaces $V_i$'s over an infinite field $k$ such that there is a $k$-epimorphism $\psi: V \to G_{e_1} \oplus \cdots \oplus G_{e_s}$ induced by the $k$-epimorphisms $\psi_i: V_i \to G_{e_i}$ for $1 \leq i \leq s$. For convenience we can assume that $i=1=j$, i.e., $y=y_{11}$. Clearly $\psi_i$ induces $\overline{\psi_i}: V_i \to \overline{G}_{e_i}$ for all $i$. Note that $\overline{\psi_i}=\psi_i$ for all $i \neq 1$. 
	By definition of general elements there exists a non-empty Zariski-open subset $U$ of $V_1^{r-1}$ such that for all $(a_1, \ldots, a_{r-1})\in \psi_1(U)$,
	\begin{equation}\label{eq0}
	\overline{F}_{\nn}= (\overline{a_1}, \ldots, \overline{a_{r_1-1}}) \overline{F}_{\nn-e_1}+ \sum_{t=2}^{s} (\overline{b_{t1}}, \ldots, \overline{b_{tr_t}})\overline{F}_{\nn-e_t}
	\end{equation}
	holds. 
	Let $z \in V_1$ such that $\psi_1(z)=y$ (as $y \neq 0$ so $z \neq 0$). Set $U'=kz \backslash \{0\}$. Then $U'$ is a non-empty Zariski open subsets of $V_1$. So by Remark \ref{rmk2}, $U'\times V_1^{r-1}$ and $V_1 \times U$ are non-empty Zariski open subsets of $V_1 \times V_1^{r-1} \simeq V_1^r$ and hence $(U' \times V_1^{r-1}) \cap (V_1 \times U)= U' \times U$ is a non-empty Zariski open subset of $V_1^r$. Note that $G/cyG=\overline{G}$ and $F/cyF=\overline{F}$ for any $0 \neq c \in k.$ Thus if we replace $y$ by $cy$ for any $0 \neq c \in k$, then also \eqref{eq0} holds. This implies that for any $(cy,a_1,\ldots,a_{r-1}) \in \psi_1(U' \times U)$, 
	\begin{equation}\label{eq00}
	F_{\nn}= (cy, a_1, \ldots, a_{r_1-1}) F_{\nn-e_1} + \sum_{t=2}^{s} (b_{t1}, \ldots, b_{tr_t})F_{\nn-e_t},
	\end{equation}  
	holds. Hence $cy, a_1, \ldots, a_{r_1-1} \in G_{e_i}$  are general elements. For $ 2 \le t \le s,$ again by the definition of general elements there exists a non-empty Zariski-open subset $U_t$ of $V_t^{r}$ such that for all $(\overline{b_{t1}}, \ldots, \overline{b_{tr_t}}) \in \overline{{\psi_t}(U_t)}$, \eqref{eq0} holds which implies that for all $(b_{t1}, \ldots, b_{tr_t}) \in \psi_t(U_t)$, \eqref{eq00} holds. Thus $b_{t1}, \ldots, b_{tr_t}$ are $r_t$ general elements in $G_{e_t}$.
\end{proof}

\begin{theorem}\label{esthm-mul-fil}
	Let $(R, \m, k)$ be a Noetherian local ring with infinite residue field $k$ and let $\f=\{I_{\underline{n}}\}_{\nn \in \N^s}$ be an $\N^s$-graded good filtration in $R.$ 
	Suppose 
	\[\mu(I_{\underline{n}})< \binom{n_1+r_1}{r_1} \cdots \binom{n_s+r_s}{r_s}\] 
	for some integers $n_1+\cdots+n_s \geq 1$ and $r_1+\cdots+r_s \geq 0$. Let $\underline{a}$ be the maximum of the degree of the generators of $F=F(\mathcal{F})$ as a module over $G=F(I_{e_1}, \ldots, I_{e_s}),$  where the maximum is taken component-wise. If $\underline{n} \geq \underline{a}+\underline{1}$, then for all $1 \leq i \leq s$ there exist ``general" elements $x_{i1}, \dots, x_{ir_i} \in I_{e_i}$ such that \[I_{\underline{n}}= \sum_{i=1}^s (x_{i1}, \dots, x_{ir_i})I_{\underline{n}-e_i}.\]
\end{theorem}

Set $G=F(I_{e_1}, \ldots, I_{e_s})= \bigoplus_{\underline{m} \in \N^s} I_{e_1}^{m_1}\cdots I_{e_s}^{m_s}/\m I_{e_1}^{m_1}\cdots I_{e_s}^{m_s}=\bigoplus_{\mm \in \N^s}\II^{\mm}/\m \II^{\mm}$, the fiber cone of the ideals $I_{e_1}, \ldots, I_{e_s}$ where $\II=(I_{e_1}, \ldots, I_{e_s})$ and $F=F(\f)= \bigoplus_{\underline{m} \in \N^s} I_{\underline{m}}/\m I_{\underline{m}}$, the fiber cone of the filtration $\f$. Then $G$ is a standard $\N^s$-graded $k$-algebra and $F$ is a finitely generated $G$-module. Again $G_{e_i}=F_{e_i}$ for all $i$ and $G_{\underline{0}}=F_{\underline{0}}=k$. 
Now \[\mu(I_{\underline{n}})< \binom{n_1+r_1}{r_1}\cdots \binom{n_s+r_s}{r_s}\] implies that $\dim_k F_{\underline{n}}<\binom{n_1+r_1}{r_1}\cdots \binom{n_s+r_s}{r_s}$. Set $V_i=I_{e_i}/\m I_{e_i}$ for all $1 \leq i \leq s$ and $V= V_1 \times \cdots \times V_s$. Since $R$ is Noetherian, $V_1, \ldots, V_s$ are finite dimensional vector spaces. Note that $V=V_1 \times \cdots \times V_s= I_{e_1}/\m I_{e_1} \times \cdots \times I_{e_s}/\m I_{e_s}= I_{e_1}/\m I_{e_1} \oplus \cdots \oplus I_{e_s}/\m I_{e_s}= G_{e_1} \oplus \cdots \oplus G_{e_s}$. Using graded Nakayama Lemma, to prove the foregoing theorem it is enough to prove the following result.

\begin{proposition}\label{mul-fil-pro}
	Let $V= V_1 \times \cdots \times V_s$ be a Cartesian product of finite-dimensional $k$-vector spaces $V_1, \ldots, V_s$ over an infinite field $k$. Let $F= \bigoplus_{\underline{n} \in \N^s} F_{\underline{n}}$ be an $\N^s$-graded algebra over $k$ and be a finitely generated $G=\bigoplus_{\underline{m} \in \N^s} G_{\underline{m}}$-module which is a standard $\N^s$-graded $k$-algebra such that $G_{\underline{0}}=F_{\underline{0}}=k$ and $G_{e_i}=F_{e_i}$ 
	for all $i$. Let $\underline{a}$ be the maximum of the degree of the generators of $F$ as a $G$-module,  where the maximum is taken component-wise.
	Suppose there is a $k$-epimorphism $\psi: V \to G_{e_1} \oplus \cdots \oplus G_{e_s}$ induced by the $k$-epimorphisms $\psi_i: V_i \to G_{e_i}$ for $1 \leq i \leq s$. Further let for some integers $n_1+\cdots+n_s \geq 1$ and $r_1+\cdots+r_s \geq 0$, $\dim_k F_{\underline{n}}<\binom{n_1+r_1}{r_1}\cdots \binom{n_s+r_s}{r_s}$. If $\underline{n} \geq \underline{a}+\underline{1}$, then there exist ``general" elements $x_{i1}, \dots, x_{ir_i} \in G_{e_i}$ such that
	\[F_{\underline{n}}= \sum_{i=1}^s (x_{i1}, \dots, x_{ir_i})F_{\underline{n}-e_i}.\]
\end{proposition}

\begin{proof}
	If $r_1+\cdots+r_s=0$ then $r_i=0$ for all $i$. So we get $\dim_k F_{\underline{n}}<1$, i.e., $F_{\underline{n}}=0$. Since each $r_i=0$, $(x_{i1},\ldots,x_{ir_i})=0$ by convention. So the result follows. Again if $n_1+\cdots+n_s=1$ then $n_i=1$ for some $i$ and $n_j=0$ for all $j \neq i$. Then $\dim_k F_{e_i}< \binom{1+r_i}{r_i}$, i.e., $\dim_k G_{e_i}< \binom{1+r_i}{r_i}$ and hence by \cite[Theorem 2.1]{caviglia} the result follows (as $G_{\underline{0}}=F_{\underline{0}}$).
	So we may assume that $r_1+\cdots+r_s \geq 1$ and $n_1+\cdots+n_s \geq 2$. 
	
	Now suppose that the result is false. Choose a counter example $F= \bigoplus_{\underline{n} \in \N^s} F_{\underline{n}}$. Pick $r_1,\ldots, r_s$ such that $r_1+\cdots+r_s$ is minimal and $n_1+\cdots+n_s$ is minimal for the chosen $r_1,\ldots, r_s$. Let $y_{i1}, \ldots, y_{ip_i}$ be a basis of $G_{e_i}$ for all $i$. Then clearly $\{y_{i1}, \ldots, y_{ip_i}\}_{i=1}^s$ forms a basis for $G_{e_1} \oplus \cdots \oplus G_{e_s}$. Without loss of generality we may assume that $r_1 \geq 1$ and $n_1 \geq 1$.
	
	\noindent
	{\bf Case-1:}
	There exists $j \in \{1, \ldots, p_1\}$ such that \[\dim_k ~y_{1j} F_{\underline{n}-e_1} \geq \binom{n_1+r_1-1}{r_1}\binom{n_2+r_2}{r_2}\cdots \binom{n_s+r_s}{r_s}. \]
	Note that by given condition $y_{1j} F_{\underline{n}-e_1} \subseteq F_{\underline{n}}$. Without loss of generality we may assume that $j=1$. As $y_{11}$ is a homogeneous element, we can pass to the factor ring $\overline{F}= F/y_{11}F$.  So for all $\nn$, we get 
	\begin{align*}
	\dim_k \overline{F}_{\nn}
	=& \dim_k F_{\nn}- \dim_k y_{11}F_{\nn-e_1}\\
	<& \binom{n_1+r_1}{r_1}\binom{n_2+r_2}{r_2}\cdots \binom{n_s+r_s}{r_s}-\binom{n_1+r_1-1}{r_1}\binom{n_2+r_2}{r_2}\cdots \binom{n_s+r_s}{r_s}\\
	=&\binom{n_1+r_1-1}{r_1-1}\binom{n_2+r_2}{r_2}\cdots \binom{n_s+r_s}{r_s}.
	\end{align*}
	Set $\overline{G}=G/y_{11}G$. Clearly $\overline{G}$ is standard graded and $\overline{F}$ is a finitely generated $\overline{G}$-module. The natural map $\nu: G \to \overline{G}$ induces the $k$-vector space epimorphism $\overline{\psi}:= \nu \circ \psi: V \to \overline{G}_{e_1} \oplus \cdots \oplus \overline{G}_{e_s}$. Note that $\overline{G}_{e_1}= G_{e_1}/y_{11}G_{\underline{0}}$ and $\overline{G}_{e_i}= G_{e_i}$ for all $i \neq 1$. Set $\nu|_{G_{\n}}=[\nu]_{\n}: G_{\n} \to \overline{G}_{\n}$. Clearly $\overline{\psi}$ is induced by $\overline{\psi}_i:= [\nu]_{e_i} \circ \psi_i: V_i \to G_{e_i} \to \overline{G}_{e_i}$ for all $i$. Now by minimality of $r_1+\cdots+r_s$, there exists general elements $a_1, \ldots, a_{r_1-1} \in G_{e_1}$ and $b_{i1}, \ldots, b_{ir_i} \in G_{e_i}$ for all $i \neq 1$ such that 
	\begin{equation}\label{eq3}
	\overline{F}_{\nn}= (\overline{a_1}, \ldots, \overline{a_{r_1-1}}) \overline{F}_{\nn-e_1}+ \sum_{i=2}^{s} (\overline{b_{i1}}, \ldots, \overline{b_{ir_i}})\overline{F}_{\nn-e_i}.
	\end{equation}
	If $r_1=1$, then \eqref{eq3} implies that $F_{\nn}= (y_{11}) F_{\nn-e_1}+ \sum_{i=2}^{s} (b_{i1}, \ldots, b_{ir_i})F_{\nn-e_i}.$ Note that this equality is true even if we replace $y_{11}$ by $c y_{11}$ for any $c \in k \backslash \{0\},$ which is a Zariski open subset of $G_{e_1}.$ Thus $y_{11} \in G_{e_1}$ is a general element.
	If $r_1\ge 2$, then by Lemma \ref{lem1} it follows that $y_{11}, a_1, \ldots, a_{r_1-1}\in G_{e_1}$ and $b_{i1}, \ldots, b_{ir_i} \in G_{e_i}$ for $i \neq 1$ are general elements satisfying \[F_{\nn}= (y_{11}, a_1, \ldots, a_{r_1-1}) F_{\nn-e_1}+ \sum_{i=2}^{s} (b_{i1}, \ldots, b_{ir_i})F_{\nn-e_i}.\]  Hence in both cases, we arrive at a contradiction to our assumption.

	\noindent
	{\bf Case-2:}
	For each $i$ and $j_i \in \{1, \ldots, p_i\}$; 
	\begin{equation}\label{eq1}
	\dim_k ~y_{ij_i} F_{\nn-e_i}< \binom{n_1+r_1}{r_1}\cdots\binom{n_i+r_i-1}{r_i}\cdots \binom{n_s+r_s}{r_s}.
	\end{equation}
	%For each $i$ and $j_i$, we have a degree $e_i$ isomorphism $F_{\nn}/ \ann_F y_{ij_i} \simeq y_{ij_i}F_{\nn}$ induced by multiplication of $y_{ij_i}$. 
	Set $K^{(ij_i)}= \ann_F y_{ij_i}$, $L^{(ij_i)}= \ann_G y_{ij_i}$ for $1 \leq i \leq s$ and $j_i \in \{1, \ldots, p_i\}$. Notice that all $K^{(ij_i)}$, $L^{(ij_i)}$ are homogeneous ideals and $L^{(ij_i)} F \subseteq K^{(ij_i)}$. Set $F^{(ij_i)}=F/K^{(ij_i)}$ and $G^{(ij_i)}=G/L^{(ij_i)}$. Clearly $F^{(ij_i)}$ is a finitely generated $G^{(ij_i)}$-module. Then for each $i, j_i$ we write \[K^{(ij_i)}= \bigoplus_{\nn \in \N^s} K^{(ij_i)}_{\nn}, \ F^{(ij_i)}=\bigoplus_{\nn\in \N^s} F^{(ij_i)}_{\nn}, \
	L^{(ij_i)}= \bigoplus_{\nn\in \N^s} L^{(ij_i)}_{\nn}  \mbox{ and }  G^{(ij_i)}=\bigoplus_{\nn\in \N^s} G^{(ij_i)}_{\nn},\] 
	using the natural multi-grading. Clearly $F^{(ij_i)}_{\nn}= F_{\nn}/K^{(ij_i)}_{\nn}$ and $G^{(ij_i)}_{\nn}= G_{\nn}/L^{(ij_i)}_{\nn}$. Note that $K^{(ij_i)}_{e_i}= \ann_{F_{e_i}} y_{ij_i}$ for all $i$ and $j_i$. So for each $i$ and $j_i$ we get a degree $e_i$ isomorphism $y_{ij_i}F_{\nn-e_i} \simeq F^{(ij_i)}_{\nn-e_i}$. Moreover, for each $i$ and $j_i$, the natural map $\nu^{(ij_i)}: G \to G^{(ij_i)}$ induces the $k$-vector space epimorphism $\psi^{(ij_i)}:=\nu^{(ij_i)} \circ \psi : V \to G^{(ij_i)}_{e_1} \oplus \cdots \oplus G^{(ij_i)}_{e_s}$. Clearly $\psi^{(ij_i)}$ can be induced from the epimorphisms 
	\begin{align*}
	\psi^{(ij_i)}_t&:= [\nu^{(ij_i)}]_{e_t} \circ \psi_t: V_t \to G_{e_t} \to G^{ij_i}_{e_t}
	\end{align*}
	for $1 \leq t \leq s$. Note that $n_1+\cdots+(n_i-1)+\cdots+n_s< n_1+\cdots+n_s$ in \eqref{eq1}. By minimality of $n_1+\cdots+n_s$ for given $r_1+\cdots+r_s$ and by the meaning of ``general" for any $i,j_i$ there exists non-empty Zariski-open subset $U^{ij_i}_t$ of $V_t^{r_t}$ yielding elements $z^{ij_i}_{t1}, \ldots, z^{ij_i}_{tr_t} \in \psi_t(V_t)$ for all $1 \leq t \leq s$ such that 
	\[F_{\nn-e_i}^{(ij_i)}= \sum_{t=1}^s (\overline{z^{ij_i}_{t1}}, \ldots, \overline{z^{ij_i}_{tr_t}})F^{(ij_i)}_{\underline{n}-e_i-e_t}.\] 
	Observe that if there exists $u \in \{1,\ldots,s\}$ such that $n_u=0$, then the above equation holds clearly. Thus $F_{\nn-e_i}= \sum_{t=1}^s (z^{ij_i}_{t1}, \ldots, z^{ij_i}_{tr_t})F_{\underline{n}-e_i-e_t}+ K^{ij_i}_{\nn-e_i}$ and hence \[y_{ij_i}F_{\nn-e_i}= \sum_{t=1}^s (z^{ij_i}_{t1}, \ldots, z^{ij_i}_{tr_t})y_{ij_i}F_{\underline{n}-e_i-e_t} \subseteq \sum_{t=1}^s (z^{ij_i}_{t1}, \ldots, z^{ij_i}_{tr_t})F_{\underline{n}-e_t}.\]
	Set $U_t= \bigcap_{i=1}^s \left(\bigcap_{j_i=1}^{p_i} U^{ij_i}_t\right)$. By Remark \ref{rmk2} we get that $U_t$ is a non-empty Zariski-open subset of $V_t^{r_t}$ for all $t$ independent of $i$ and $j_i$ such that for the corresponding elements $x_{t1}, \ldots, x_{tr_t} \in \psi_t(V_t)=G_{e_t}$,
	\begin{align} \label{thm3.2eq}
	\sum_{i=1}^s \left(\sum_{j_i=1}^{p_i}y_{ij_i}\right)F_{\underline{n}-e_i} \subseteq \sum_{t=1}^s (x_{t1}, \ldots, x_{tr_t})F_{\underline{n}-e_t} \subseteq F_{\nn}.
	\end{align}

	Let $\{g_1, \ldots, g_v\}$ be the set of generators of $F$ as a $G$-module, with $\deg g_i=\underline{a_i}$ where $\underline{a_i}=(a_{i1}, \ldots, a_{is})\in \N^s$. Then for all $\underline{m}$, 
	\[F_{\underline{m}}=g_1 G_{\underline{m}-\underline{a_1}}+ \cdots+g_v G_{\underline{m}-\underline{a_v}}.\] 
	Note that $G_{\underline{m}-\underline{a_i}}=0$ if $m_j< a_{ij}$ for some $1 \leq j \leq s$. As $G$ is a standard $\N^s$-graded, we have 
	\[G_{\underline{m}}= \sum_{i=1}^s\left(\sum_{j_i=1}^{p_i} y_{ij_i}\right)G_{\underline{m}-e_i} \] 
	for all $\underline{m} \geq \underline{1}.$ Set $\underline{a}=\max\{\underline{a_1}, \ldots, \underline{a_v}\}$, where the maximum is taken component-wise. If $\underline{n}\geq \underline{a}+\underline{1}$, then 
	\begin{align*}
	F_{\underline{n}}
	=g_1 G_{\underline{n}-\underline{a_1}}+ \cdots+g_v G_{\underline{n}-\underline{a_v}}
	&=\sum_{u=1}^v g_{u} \left(\sum_{i=1}^s\left(\sum_{j_i=1}^{p_i} y_{ij_i}\right)G_{\underline{n}-\underline{a_u}-e_i}\right)\\
	&=\sum_{i=1}^s\left(\sum_{j_i=1}^{p_i} y_{ij_i}\right) \left(\sum_{u=1}^v g_{u} G_{\underline{n}-\underline{a_u}-e_i}\right)\\
	&=\sum_{i=1}^s\left(\sum_{j_i=1}^{p_i} y_{ij_i}\right)F_{\underline{n}-e_i}.
	\end{align*}
	Thus if $\underline{n} \geq \underline{a}+\underline{1}$, then using the above equation and \eqref{thm3.2eq}, we get 
	\[ F_{\underline{n}} = \sum_{t=1}^s (x_{t1}, \ldots, x_{tr_t})F_{\underline{n}-e_t} \]
	a contradiction to our assumption. Hence the result follows.
\end{proof}

\begin{remark}
	Let $\f=\{I_n\}_{n \geq 0}$ be a good filtration such that $\mu(I_n)< \binom{n+r}{r}$ for some $n>0$ and $r \geq 0.$ If $n \geq a+1$, where $a$ is as defined in the above theorem, then by Theorem \ref{esthm-mul-fil} there exist $r$ general elements $x_1, \ldots, x_r \in I_1$ such that $I_{n}=(x_1, \ldots, x_r)I_{n-1}.$ Let $\dim F(\f)=r$. If $I_{m}=(x_1, \ldots, x_r)I_{m-1}$ for all $m \geq n$, then $(x_1, \ldots, x_r)$ is a minimal reduction of $\f.$ But $I_{n}=(x_1, \ldots, x_r)I_{n-1}$ for some $n$ does not always imply $I_{m}=(x_1, \ldots, x_r)I_{m-1}$ for all $m \geq n$. Later in Theorem \ref{ESf} we give a sufficient condition for $(x_1, \ldots, x_r)$ being a minimal reduction of $\f$. 
\end{remark}

As an immediate consequence of the above result, we get the following.
\begin{corollary}\label{mulg-adic}
	Let $(R, \m, k)$ be a Noetherian local ring with infinite residue field $k$ and $I_1,\ldots, I_s$ be ideals in $R$. Suppose 
	\[\mu(I_1^{n_1}\cdots I_s^{n_s})< \binom{n_1+r_1}{r_1}\cdots \binom{n_s+r_s}{r_s}\] 
	for some integers $n_1+\cdots+n_s \geq 1$ and $r_1+\cdots+r_s \geq 0$. Then for all $1 \leq i \leq s$ there exist ``general" elements $x_{i1}, \dots, x_{ir_i} \in I_i$ such that for all $\underline{m} \geq \underline{n}$, 
	\[I_1^{m_1}\cdots I_s^{m_s}= \sum_{i=1}^s (x_{i1}, \dots, x_{ir_i})I_1^{m_1}\cdots I_i^{m_i-1}\cdots I_s^{m_s}.\]
\end{corollary}

\section{Eakin-Sathaye theorem for reduction number of $\N$-graded good filtrations}

We now prove an analogue of the Eakin-Sathaye Theorem to estimate the reduction number of an equimultiple good $\N$-graded filtration in Cohen-Macaulay rings. We impose additional assumptions on depth of the associated graded ring and the fiber cone.

Let $R$ be a Noetherian local ring and $\f=\{I_n\}$ be a filtration. For an element $x \in I_1,$ $x^* \in I_1/I_2$ denotes the image of $x$ in $\gr_{\f}(R)$ and $x^\circ \in I_1/\m I_1$ denotes the image of $x$ in $F(\f).$ If $x^* \not= 0$, then it is said to be superficial in $\gr_{\f}(R)$ if $(0 : x^*) \cap \gr_{\f}(R)_n = 0$ for all $n$ large. Similarly, if $x^\circ \not= 0,$ then it is superficial in $F(\f)$ if $(0 : x^\circ) \cap F(\f)_n = 0$ for all $n$ large. Let $\min(R)$ denote the set of all minimal prime ideals of $R.$

\begin{lemma}\label{depthff}
	Let $(R,\m)$ be a Noetherian local ring of dimension $d>0$ with $R/\m=k$ infinite and $I$ be an ideal in $R$ such that $I \nsubseteq \mathfrak{p}$ for any $\mathfrak{p} \in \min(R)$. Let 
	$\f=\{I_n\}_{n \in \N}$ be an $I$-good filtration. Then there exists $x \in I \setminus \m I_1$ such that:\\
	{\rm (1)} $x \notin \bigcup_{\fp \in \min(R)} \fp$ \\
	{\rm (2)} $x^\circ$ is superficial in $F(\f)$ \\ 
	{\rm (3)} $x^*$ is superficial in $\gr_{\f}(R)$.
\end{lemma}

\begin{proof}
	Let  
	\begin{align*}
	\Ass (\gr_{\f}(R)) = \{P_1, \ldots, P_r, P_{r+1} \ldots, P_{r'}\}, \quad 
	\Ass (F(\f)) = \{Q_1, \ldots, Q_m, Q_{m+1}, \ldots, Q_{m'}\} 
	\end{align*} 
	such that for all $n$ large, $I_n/I_{n+1} \subseteq P_i$ for $r+1 \le i \le r'$ and $I_n/\m I_n \subseteq Q_j$ for $m+1 \le j \le m'.$ Consider the ideal $\mathscr{I}= \oplus_{n \ge 0}I_{n+1}t^n$ of $\R(\f).$ Since $\gr_{\f}(R) = \mathcal{R}(\f)/\mathscr{I}$ and $F(\f) = \mathcal{R}(\f)/\m\mathcal{R}(\f)$, let $\mathcal{P}=\{P'_1, \ldots, P'_r, Q'_1, \ldots, Q'_m\}$ be the collection of prime ideals in $\mathcal{R}(\f)$ which are pre-images of the ideals $P_1,\ldots,P_r$ and $Q_1,\ldots,Q_m.$ Observe that $\mathcal{R}(I) \subseteq \mathcal{R}(\f)$ is an integral extension. Set $P'_i \cap \mathcal{R}(I)=P''_i$ and $Q'_j \cap \mathcal{R}(I)=Q''_j$ for all $i$ and $j$. Since $k$ is infinite, $I/\m I \not= V$, where
	\begin{align*}
	V = \left(\frac{\m I_1 \cap I}{\m I}\right) 
	\bigcup \left(\frac{I_2 \cap I + \m I}{\m I}\right)
	\bigcup_{\fp \in \min(R)} \left(\frac{\fp \cap I + \m I}{\m I}\right) \
    \bigcup_{i=1}^r \left(\frac{P''_i \cap I + \m I}{\m I}\right) \
	\bigcup_{j=1}^m \left(\frac{Q''_j \cap I + \m I}{\m I}\right).
	\end{align*} 
	Hence we can choose 
	\[x \in I \setminus \left( (\m I_1 \cap I) \bigcup \ (I_2 \cap I) \bigcup_{\fp \in \min(R)} (\fp \cap I) \ \bigcup_{i=1}^r (P''_i \cap I) \ \bigcup_{j=1}^m(Q''_j \cap I) \right).\] 
	We first show that $(0 : x^\circ) \cap F(\f)_n = 0$ for $n$ large. Let $y^\circ \in (0 : x^\circ).$ Write $(0) = M_1 \cap \cdots \cap M_{m'}$ be a primary decomposition of $(0)$ in $F(\f)$ such that $M_j$ is $Q_j$-primary for $j =1,\dots,m'.$ Then $y^\circ x^\circ \in M_j$ for all $1 \le j \le m'.$ Since $x^\circ \notin Q_j$ for $j =1,\dots,m$, it follows that $y^\circ \in M_j$ for $j =1,\dots,m.$ Thus $(0 : x^\circ) \subseteq M_1 \cap \cdots \cap M_m.$ For $m + 1 \le j \le m'$, $F(\f)_n \subseteq Q_j$ for $n$ large. Therefore $F(\f)_n \subseteq M_{m+1} \cap \cdots \cap M_{m'}.$ This implies that for all $n$ large, $(0 : x^\circ) \cap F(\f)_n \subseteq M_1 \cap \cdots \cap M_{m'} = (0).$ Hence $x^\circ$ is superficial in $F(\f).$ A similar argument shows that $x^*$ is superficial in $\gr_{\f}(R).$
\end{proof}

\begin{remark}\label{rmksup}
	In the above proof, observe that there exists a non-empty Zariski open subset $U$ of $I/\m I$, such that for any $x+\m I \in U$, the lemma holds. Thus $x \in I$ is a general element.
\end{remark}

\begin{theorem}\label{ESfdim1}
	Let $(R, \m)$ be a Cohen-Macaulay local ring of dimension $d>0$ with $R/\m=k$ infinite. Let $\f=\{I_n\}_{n \in \N}$ be an equimultiple good filtration such that $\grade\gr_{\f}(R)_+ \ge l(\f) =1$ and $F(\f)$ is Cohen-Macaulay. If $\mu(I_n)< (n+1)$ for some $n \ge 1$, then there exists a general element $x \in I_1$ such that  $I_m = (x)I_{m-1}$ for all $m \ge n$. 
\end{theorem}

\begin{proof}
	By Lemma \ref{depthff} and Remark \ref{rmksup} there exists a general element $x \in I_1$ such that $x \notin \bigcup_{\mathfrak{p} \in \min(R)} \mathfrak{p}$ and $x^\circ$, $x^*$ are superficial elements in $F(\f)$ and $\gr_{\f}(R)$ respectively. Since $R$ is Cohen-Macaulay and $\hgt I_1=1$ so $x$ is a nonzerodivisor. As $\depth F(\f)= 1$, all the associated primes of $(0)$ in $F(\f)$ are relevant primes and hence $x^\circ$ is $F(\f)$-regular. Using \cite[Lemma 2.1]{huckabaMarley}, $x^*$ is $\gr_{\f}(R)$-regular. From \cite[Proposition 3.5]{huckabaMarley} it then follows that $(x) \cap I_i = (x)I_{i-1}$ for all $i \ge 1$. Note that $I_{i-1} \simeq (x)I_{i-1}$ which implies that $\mu((x)I_{i-1})=\mu(I_{i-1})$ for all $i \ge 1$. Since $\depth_{(x^\circ)} F(\f)=1$, using \cite[Theorem 2.8]{cortadellasZar} it follows that $(x) \cap \m I_i= (x)\m I_{i-1}$ for all $i \geq 1$. 
	Therefore,
	$$(x)I_{i-1} \cap \m I_i \subseteq (x) \cap \m I_i=(x)\m I_{i-1} \subseteq (x)I_{i-1} \cap \m I_i$$ 
	and hence $(x)I_{i-1} \cap \m I_i= (x)\m I_{i-1}$ for all $i \ge 1$. For all $i \ge 1$,
	\begin{align*}
	\mu(I_i) - \mu(I_{i-1})
	= \mu(I_i)-\mu((x)I_{i-1})
	&= \dim_k \frac{I_i}{\m I_i}- \dim_k \frac{(x)I_{i-1}}{\m(x) I_{i-1}} \\
	&= \dim_k \frac{I_i}{\m I_i}- \dim_k \frac{(x)I_{i-1}}{(x)I_{i-1} \cap \m I_i}\\
	&= \dim_k \frac{I_i}{\m I_i}- \dim_k \frac{(x)I_{i-1}+\m I_i}{\m I_i} \\
	&= \dim_k \frac{I_i}{(x)I_{i-1}+\m I_i} \\
	&= \ell(I_i/\m I_i+(x)I_{i-1}) \geq 0
	\end{align*}
	Note that $\ell(I_i/\m I_i+(x)I_{i-1})=0$ if and only if $I_i=\m I_i+(x)I_{i-1}$, i.e., $I_i=(x)I_{i-1}$, for all $i \ge 1$, by Nakayama Lemma. 
	
	For all $i \ge 1$, the inclusion map $f_i: I_{i+1} \hookrightarrow I_{i}$ induces the map $$\widetilde{f}_i : \frac{I_{i+1}}{(x)I_{i}} \to \frac{I_{i}}{(x)I_{i-1}}.$$ 
	We claim that $\widetilde{f}_i$ is an injective map for all $i.$ It is sufficient to show that $(x)I_{i-1} \cap I_{i+1}=(x)I_{i}$. Clearly $(x)I_{i} \subseteq (x)I_{i-1} \cap I_{i+1}$. Let $xy \in (x)I_{i-1} \cap I_{i+1}$ for some $y \in I_{i-1}$. Then $y \in (I_{i+1}: x)=I_{i}$ as $x^*$ is $\gr_{\f}(R)$-regular. Thus $xy \in (x)I_{i}$ and hence $(x)I_{i-1} \cap I_{i+1} \subseteq (x)I_{i}$. The claim follows and hence $\widetilde{f}_i$ is injective for all $i \ge 1$. It follows that if $I_i = (x)I_{i-1}$ for some $i \ge 1$, then $I_j = (x)I_{j-1}$ for all $j \geq i$. 
	
	If possible, let $I_n \neq  (x)I_{n-1}$. Then by the above observation $I_i \neq (x)I_{i-1}$ for all $1 \le i \leq n$. Thus we have $0< \mu((x))< \mu(I_1)< \cdots< \mu(I_n)$ and hence $\mu(I_n)\geq n+1$, a contradiction. Therefore $I_n =  (x)I_{n-1}$ which implies that $I_m = (x)I_{m-1}$ for all $m \geq n$. 
\end{proof}

\begin{remark}
	In the above theorem, if $n=1$, i.e., $\mu(I_1)<2$, then $I_m = (x)I_{m-1}$ for all $m \geq 1$. In particular, $(x) = I = I_1$ and hence $I_n = (x^n)$ for all $n.$
\end{remark}

\begin{theorem}\label{ESf}
	Let $(R, \m)$ be a Cohen-Macaulay local ring with $R/\m=k$ infinite. Let $\f=\{I_n\}_{n \in \N}$ be an equimultiple good filtration such that $\grade\gr_{\f}(R)_+ \geq l(\f)=r$ and $F(\f)$ is Cohen-Macaulay. Let $\mu(I_n)< \binom{n+r}{r}$ for some $n\ge 1$. Then there exist $r$ general elements $x_1, \ldots, x_r \in I_1$ such that $I_m = (x_1,\ldots,x_r)I_{m-1}$ for all $m \geq n$. 
\end{theorem}

\begin{proof}
	If $r=0$, then $(0)$ is the minimal reduction of $\f$. Now $\mu(I_n)< \binom{n+0}{0}=1$ implies $I_n=(0)$ and hence $I_m=(0)$ for all $m \geq n$. Therefore $I_m=(0)I_{m-1}$ for all $m \geq n$. Note that if $r = 1$, then $\dim R \ge r = 1$. Thus the result follows from Theorem \ref{ESfdim1}. Therefore we may assume that $r \geq 2$. 

	Suppose the result is false. Choose a counter example $(R, \m)$ in which $r$ is minimal and $n$ is minimal for this given value of $r$. Let $\dim R= d \geq 0$. If $d = 0$, then $r = 0$ and if $d = 1$, then $r = 0$ or $1$. We have seen that in all cases the result holds. So we may assume that $d \geq 2$.
	As $R$ is Cohen-Macaulay, by Lemma \ref{depthff} there exists a nonzerodivisor $a \in I_1$ such that $a^\circ$ and $a^*$ are superficial in $F(\f)$ and $\gr_{\f}(R)$ respectively. Using \cite[Lemma 2.1]{huckabaMarley}, $a^*$ is $\gr_{\f}(R)$-regular and from \cite[Proposition 3.5]{huckabaMarley} it then follows that $(a) \cap I_n= a I_{n-1}$ for all $n \ge 1.$ Since $\depth F(\f) \ge 2$, $a^\circ$ is a nonzerodivisor in $F(\f)$ so \cite[Theorem 2.8]{cortadellasZar} implies that $(a) \cap \m I_n= a \m I_{n-1}$ for all $n \ge 1.$ 

	Set $\overline{R}=R/(a)$ and $\overline{I_n}$ to be the image of $I_n$ in $\overline{R}$. Then  
	\begin{align*}
	\frac{\overline{I_n}}{\m \overline{I_n}} 
	= \frac{I_n+(a)}{\m I_n+(a)}
	= \frac{I_n+(\m I_n+(a))}{\m I_n+(a)}
	=\frac{I_n}{\m I_n+\left((a)\cap I_n \right)}
	=\frac{I_n}{\m I_n+ aI_{n-1}}.
	\end{align*}
	Thus for all $n \ge 1$,
	\begin{align*}
	\mu(\overline{I_n})
	&= \dim_k ~I_n/ \left(\m I_n+ aI_{n-1}\right)\\
	&=\dim_k ~I_n/\m I_n- \dim_k \left(\m I_n+ aI_{n-1}\right)/ \m I_n\\
	&=\dim_k ~I_n/\m I_n- \dim_k ~ aI_{n-1}/ \left(\m I_n \cap aI_{n-1}\right).
	\end{align*}
	Now 
	$$\m I_n \cap aI_{n-1}= \m I_n \cap (a) \cap aI_{n-1}= a\m I_{n-1} \cap a I_{n-1}=a\m I_{n-1}$$ 
	and hence $\mu(\overline{I_n})= \dim_k ~I_n/\m I_n- \dim_k ~ aI_{n-1}/a\m I_{n-1}$ for all $n \ge 1$. Note that $aI_{n-1}/a\m I_{n-1} \simeq I_{n-1}/\m I_{n-1}$ as $a$ is a nonzerodivisor  in $R$. This implies that $\mu(\overline{I_n})= \mu(I_n)-\mu(I_{n-1})$ for all $n \ge 1$.

	{\bf Case 1}: $\mu(I_{n-1})< \binom{n-1+r}{r}$.

	By minimality of $n$ for the chosen $r$, there exist $r$ general elements $x_1, \ldots, x_r \in I_1$ such that $I_m = (x_1,\ldots,x_r)I_{m-1}$ for all $m \geq n-1$, a contradiction.

	{\bf Case 2}: $\mu(I_{n-1})\geq \binom{n-1+r}{r}$.

	We get $\mu(\overline{I_n})< \binom{n+r}{r}-\binom{n-1+r}{r} = \binom{n+r-1}{r-1}$. Set $\overline{\f}=\f/(a)=\{\overline{I_n}\}_{n \in \N}$. We claim that $F(\overline{\f}) \simeq F(\f)/({a}^\circ)$ and $\gr_{\overline{\f}}(\overline{R}) \simeq \gr_{\f}(R)/({a}^*)$. Indeed, 
	\begin{align*}
	F(\overline{\f}) 
	\simeq \bigoplus_{n=0}^{\infty} \frac{I_n + (a)}{\m I_n + (a)} 
	\simeq \bigoplus_{n=0}^{\infty} \frac{I_n}{\m I_n +(I_n \cap (a))}
	\simeq \bigoplus_{n=0}^{\infty} \frac{I_n}{\m I_n + (a)I_{n-1}} 
	\simeq F(\f)/({a}^\circ) .
	\end{align*}
	Similarly, 
	\begin{align*}
	\gr_{\overline{\f}}(\overline{R})
	\simeq \bigoplus_{n=0}^{\infty} \frac{I_n + (a)}{I_{n+1} + (a)} 
	\simeq \bigoplus_{n=0}^{\infty} \frac{I_n}{I_{n+1}+(I_n \cap (a))}
	\simeq \bigoplus_{n=0}^{\infty} \frac{I_n}{I_{n+1} + (a)I_{n-1}}
	\simeq \gr_{\f}(R)/({a}^*) .
	\end{align*}
	As $a^*$ and $a^\circ$ are regular elements in $\gr_{\f}(R)$ and $F(\f)$ respectively, $F(\overline{\f})$ is Cohen-Macaulay with $l(\overline{\f})=l(\f)-1=r-1 \geq 1$ and
	$$\grade \gr_{\overline{\f}}(\overline{R})_+ = \grade \gr_{\f}(R)_+ -1 \geq l(\f)-1=l(\overline{\f}).$$ 
	Since $R/(a)$ is Cohen-Macaulay of dimension $d-1$ and $(R/(a))/(I_1/(a)) \simeq R/I_1$, it follows that $\dim R/(a)-\hgt (I_1/(a))= \dim R -\hgt I_1$. So $(d-1)-\hgt (I_1/(a))=d- \hgt I_1$ and hence $\hgt (I_1/(a)) = \hgt I_1 - 1=r-1$. Thus $\overline{\f}$ is an equimultiple good filtration. By minimality of $r$, there exist $r-1$ general elements $\overline{x_1}, \ldots, \overline{x_{r-1}} \in \overline{I_1}$ such that $\overline{I_m} = (\overline{x_1},\ldots,\overline{x_{r-1}}) \overline{I_{m-1}}$ for all $m \geq n$. By our choice $0 \neq a+\m I_1 \in I_1/\m I_1=F(I_1)_1$ so by Lemma \ref{lem1} it follows that $a, x_1, \ldots, x_{r-1} \in I_1$ are $r$ general elements for which $I_m=(a, x_1, \ldots, x_{r-1}) I_{m-1}$ for all $m \geq n$, a contradiction.
\end{proof}

\section{Examples}
\subsection{\bf Contracted ideals} 

Let $(R,\m)$ be a $2$-dimensional regular local ring. An $\m$-primary ideal $I$ is called a contracted ideal \cite[App.~5]{zariskiSamuel}, if there exists an $x \in \m \backslash \m^2$ such that $IR[\m/x] \cap R=I$. 
Zariski \cite{zariski} proved that the product of contracted (complete) ideals in $R$ is contracted (complete) and a complete ideal is contracted. Set $o (I)= \m$-adic order of $I= \max \{n \mid I \subseteq \m^n\}$. Lipman \cite{lipman} and Rees \cite{rees} proved that if $I$ is  contracted then $\mu(I)=1+o(I)$, where $\mu(I)$ denotes the minimal number of generators of $I.$ Huneke-Sally \cite{hunekeSally} proved that if $R/\m$ is infinite, then the converse is also true.
Thus we have the following result.
\begin{theorem}\label{lip}
	Let $(R,\m)$ be a $2$-dimensional regular local ring with infinite residue field and $I$ be an $\m$-primary ideal in $R$. Then $I$ is contracted if and only if $\mu(I)=o(I)+1$.
\end{theorem}

Let $R$ be a $2$-dimensional regular local ring and $I,J$ be contracted ideals. Using Corollary \ref{mulg-adic}, we find a choice of the joint reduction vector $(m, n)$ such that 
\begin{equation}\label{eq2}
I^mJ^n= aI^{m-1}J^n+bI^mJ^{n-1}
\end{equation}
for some $a \in I$ and $b \in J.$

\begin{proposition}\label{jred}
	Let $(R,\m)$ be a $2$-dimensional regular local ring with infinite residue field. If $I$ and $J$ are contracted ideals, then in \eqref{eq2} we can take $m=2\cdot o(J)-1$ and $n=2\cdot o(I)-1$.
\end{proposition}

\begin{proof}
	Let $I, J$ be contracted ideals. Then $I^rJ^s$ is also a contracted ideal for any $r,s$. Set $o (I)=\alpha$ and $o(J)=\beta$. Note that $\alpha \geq 1$ and $\beta \geq 1$. Then $\mu(I^mJ^n)= o(I^mJ^n)+1= m \alpha+n \beta+1$. If we can write $\mu(I^mJ^n)< \binom{m+1}{1}\binom{n+1}{1}$ for some $m,n$, then by Corollary \ref{mulg-adic} we get equation \eqref{eq2}. So we want $(m,n)$ to be a solution of the equation 
	\begin{align*}
	\alpha x + \beta y+1< (x+1)(y+1) &\iff
	 0< xy-(\alpha-1)x -(\beta-1)y  \\
	 &\iff 0< (x-\beta+1)(y-\alpha+1)-(\alpha-1)(\beta-1)
	\end{align*}
	with $m,n \geq 0$. 
	Now take $m=2 \beta-1$ and $n=2\alpha-1$. Then we get \[(2\beta-1- \beta+1)(2\alpha-1-\alpha+1)-(\alpha-1)(\beta-1)=\alpha+\beta-1>0.\] Thus the pair $(2 \beta-1,2 \alpha-1)$ satisfies the above equation.
\end{proof}

The following example illustrates that Corollary \ref{mulg-adic} gives a better bound than the bound given by L. O'Carroll's result \cite[Corollary 3.2]{carroll}. 
\begin{example} \label{counter}
	Let $k$ be a field and $R=k[[x,y]]$ be the power series ring over $k.$ Let $I=(x, y^2)$ and $J=(y, x^2).$ Then $IJ=xJ+yI.$ Therefore the joint reduction vector of $(I,J)$ with respect to $(x,y)$ is $(1,1).$ As $IJ=(xy,x^3,y^3)$ and $I^2J^2 = (x^2y^2,x^4y,xy^4,x^6,y^6)$, it follows that $\mu(IJ)=3 \nless \binom{1+2}{2}=3$ but $\mu(I^2J^2)=5 < \binom{2+2}{2}=6$. So by \cite[Corollary 3.2]{carroll} we get 
	\[I^2J^2=bI^2J+aIJ^2\] 
	for some $a \in I$ and $b \in J$, giving the joint reduction vector to be $(2,2)$. Since $I$ and $J$ are contracted ideals, by Proposition \ref{jred} we get $IJ=bI+aJ$ for some $a \in I$ and $b \in J$, clearly giving the exact value of the joint reduction vector. 
\end{example}

\subsection{\bf Lexsegment ideals} 
\begin{definition} \cite{herzog2017}
	An ideal $I \subseteq k[x_1, \ldots, x_n]$ is called a {\it lexsegment ideal}, if for any monomial $u \in I$ and all monomials $v$ with $\deg v=\deg u$ and $v>u$ in the lexicographical order, it follows that $v \in I$.
\end{definition}

\begin{example}
	Let $I$ and $J$ be two lexsegment ideals in $R=k[x,y]$. Then by \cite[Lemma 4.3]{herzog2017}, 
	\[I=(x^r, x^{r-1}y^{b_1}, \ldots, x^{r-p}y^{b_p}) \text{ and } J=(x^s, x^{s-1}y^{a_1}, \ldots, x^{s-q}y^{a_q})\] 
	for some integers $0<b_1< \cdots<b_p$ and $0<a_1< \cdots<a_q$. Clearly $\mu(I)=p+1$ and $\mu(J)=q+1$. By \cite[Corollary 4.5]{herzog2017} it follows that 
	\[\mu(I^nJ^m)=n \mu(I)+m\mu(J)-(n+m-1)=n(p+1)+m(q+1)-(n+m-1)= pn+qm+1.\] 
	If $\mu(I^nJ^m)< \binom{n+1}{1}\binom{m+1}{1}=(n+1)(m+1)$, then by Corollary \ref{mulg-adic} there exist $a \in I$ and $b \in J$ such that $I^nJ^m=aI^{n-1}J^{m}+bI^nJ^{m-1}$, i.e., if $pn+qm+1< nm+n+m+1$, or \[(p-1)n +(q-1)m<nm.\] Note that if $p=1$ and $q=2$, then the above equation is satisfied for $n=2,m=1$ and hence joint reduction vector is $(2,1)$. If we take $p=1, q=2$ and $n=m$, then the minimum choice of $n$ such that $n<n^2$ is $2$. Notice $\mu(I^2J^2)=7 \nless \binom{2+2}{2}=6$, so L. O'Carroll's result \cite[Corollary 3.2]{carroll} is not applicable for $n=m=2$.
\end{example}

The following examples show that Theorem \ref{ESf} gives better bound for the reduction numbers of respective filtrations. We use the following proposition to characterize Cohen-Macaulay property of fiber cone. 

\begin{proposition} \cite[Proposition 3.7]{cortadellasZar} \label{cort}
	Let $\f=\{I_n\}$ be a good filtration such that $I_1$ is $\m$-primary. Assume that $\gr_{\f}(R)$ is Cohen-Macaulay and let $J$ be a minimal reduction of $\f.$ The following are equivalent: \\
	{\rm(1)} $F(\f)$ is Cohen-Macaulay. \\
	{\rm(2)} $J \cap \m I_n = J \m I_{n-1}$ for all $1 \le n \le r_J(\f).$
\end{proposition}

\begin{example}
	Let $k$ be an infinite field such that $p =\chr k \neq 3$ and $R=k[[X,Y,Z]]/(X^3+Y^3+Z^3).$ Then $R$ is a $2$-dimensional Cohen-Macaulay, analytically unramified local ring. Let $x,y,z$ denote the images of $X,Y$ and $Z$ respectively in $R$ and $I=(y,z).$ Then $\f=\{(I^n)^*\}$ is an $I$-admissible filtration and using \cite{GMV}, $(I^k)^* = \m^{k+1}+I^k$ for all $k \ge 1$ and $(I^k)^* = I(I^{k-1})^*$ for all $k \ge 2$. We show that $\gr_{\f}(R) = \oplus_{n \ge 0}(I^n)^*/(I^{n+1})^*$ and $F(\f) = \oplus_{n \ge 0}(I^n)^*/\m(I^n)^*$ are Cohen-Macaulay. 
	
	In order to show that $\gr_{\f}(R)$ is Cohen-Macaulay, using \cite[Theorem 2.3, Corollary 2.1]{viet}, it is sufficient to show that $I(I^{k-1})^* = (I^k)^* \cap I$ for all $k \ge 1$. This is true as $I(I^{k-1})^* = (I^k)^* \subseteq I$ for all $k \ge 2$. Hence, $\gr_{\f}(R)$ is Cohen-Macaulay. Using Proposition \ref{cort}, for $F(\f)$ to be Cohen-Macaulay, it is sufficient to show that $I \cap \m I^* = I\m.$ Since $\m^3 \subseteq \m I$ it follows that
	\begin{align*}
	I \cap \m I^* = I \cap \m(I+\m^2) = I \cap (\m I+ \m^3) = I \cap \m I = \m I.
	\end{align*}
	As $(I^2)^* = I^2 + \m^3 =(x^2y,x^2z,y^2,yz,z^2)$ and $\mu((I^2)^*) = 5 < \binom{2+2}{2} = 6$, Theorem \ref{ESf} implies that $r(\f) \le 1$ and hence $r(\f)=1$ as $I \neq I^*$.
\end{example}

\begin{example}\label{ex2}
	Let $R=k[X,Y,U]$ be a polynomial ring in three variables with unique homogeneous maximal ideal $\m=(X,Y,U)$ and an infinite residue field $k.$ Set $T=R_{\m}.$ Then $T$ is a regular local ring with unique maximal ideal $\m R_{\m}$ and infinite residue field $k.$ Let $I=(X^2,Y^2,U)R$ and let $\f=\{\overline{I^n}\}_{n \geq 1}$ be the integral closure filtration of $I$. Since $I$ is a homogeneous ideal, it is clear that $I\overline{I^n}=\overline{I^{n+1}}$ in $T$ if and only if $I\overline{I^n}=\overline{I^{n+1}}$ in $R$. As $T$ is an analytically unramified Noetherian local ring, by \cite[Corollary 9.2.1]{swansonHuneke} it follows that $\f$ is an $I$-good filtration in $T$. 

	We first claim that $\overline{I}=(X^2,XY,Y^2,U).$ Since $XY$ satisfies the equation $t^2-X^2.Y^2=0,$ $XY \in \overline{I}.$ It is sufficient to show that the ideal $(X^2,XY,Y^2,U)$ is integrally closed. Observe that the ideal $(X^2,XY,Y^2)$ is integrally closed in $k[X,Y] = R/(U).$ As contraction of a complete ideal is complete, the claim follows. Consider 
	\[I \overline{I}= (U^2,Y^2U, XYU, X^2U, Y^4, XY^3, X^2Y^2, X^3Y, X^4).\] 
	We have $I^2 \subseteq I \overline{I} \subseteq \overline{I^2}$. In order to show $I \overline{I}=\overline{I^2}$ it is enough to show that $I \overline{I}$ is integrally closed. Now 
	\begin{align*}
	&(U^2,Y^2U, XYU, X^2U, Y^4, XY^3, X^2Y^2, X^3Y, X^4) \\
	&= (X,U^2,Y^2U,Y^4) \cap (YU,U^2,X^2U,Y^4, XY^3, X^2Y^2, X^3Y, X^4)\\
	&= (X,U^2,Y^2U,Y^4) \cap (Y,U^2,X^2U, X^4) \cap (U,\n^4),
	\end{align*}
	where $\n=(X,Y)$ is the unique homogeneous maximal ideal of $K[X,Y]$. In view of  \cite[Proposition 1.4.6]{swansonHuneke} it follows that $(U^2,X^2U, X^4)$ and $(U^2,Y^2U,Y^4)$ are integrally closed. Thus $(X,U^2,Y^2U,Y^4)$ and $(Y,U^2,X^2U, X^4)$ are integrally closed in $R$ and hence integrally closed in $T$ by \cite[Proposition 1.1.4]{swansonHuneke}. Again $(U,\n^4)$ is integrally closed in $T$. Hence $I \overline{I}$ is integrally closed. 

	We now claim that $\mathcal{R}(\f, T)=T[\overline{I}t,\overline{I^2}t^2, \ldots]$ is Cohen-Macaulay. Let $\mathcal{R}(\f, R)$ denote the Rees ring of $R$ with respect to $\f$ and $r(\f)=n_0$. Then $\mathcal{R}(\f, R)=R[\overline{I}t,\overline{I^2}t^2, \ldots, \overline{I^{n_0}}t^{n_0}]$ which is Noetherian. Again by \cite[Proposition 1.4.2]{swansonHuneke} we have $\overline{I^n}$ is a monomial ideal for all $n \geq 1$. Therefore $M=(X,Y,U,\overline{I}t,\overline{I^2}t^2, \ldots)$ is a semi-group of monomials in $X,Y,U$. Moreover, by \cite[Proposition 5.2.4]{swansonHuneke} the integral closure of $R[It]$ in its field of fractions is $\mathcal{R}(\f, R)$. Since $R[It] \subseteq \mathcal{R}(\f, R) \subseteq \Frac(R[It])$ so the field of fractions of $\mathcal{R}(\f, R)$ is also $\Frac(R[It])$. Hence $\mathcal{R}(\f, R)=k[M] \subseteq k[X,Y,U,t]$ is normal. Thus by \cite[Proposition 1]{hochster} $M$ is a normal semigroup of monomials and by \cite[Theorem 1]{hochster} it follows that $\mathcal{R}(\f, R)$ is Cohen-Macaulay. 
	Since $R \subseteq \mathcal{R}(\f, R)$ is a subring so $C=R-\m$ is also a multiplicative closed subset of $\mathcal{R}(\f, R)$. Hence $\mathcal{R}(\f, T)= C^{-1}\mathcal{R}(\f, R)$ is Cohen-Macaulay. Thus by \cite[Corollary 2.1]{viet} we get that $\gr_{\f}(T)$ is Cohen-Macaulay.

	Using \cite[Theorem 2.3]{viet} it follows that $r(\f) \leq \dim T-1=2$. Since $I$ is a minimal reduction of $\f$, $I \overline{I^{n}}=\overline{I^{n+1}}$ for all $n \geq 2$. Thus in our case, $r_I(\f)=1$ (as $I \neq \overline{I}$). Now 
	\[ I \cap (X,Y,U)\overline{I}= I (X,Y,U)=(U^2, YU, XU, Y^3, XY^2, X^2Y, X^3) \]
	and hence by Proposition \ref{cort}, it follows that $F(\f)$ is Cohen-Macaulay. Thus our assumptions in Theorem \ref{ESf} are satisfied and as $\mu(\overline{I^2})=9< \binom{2+3}{3}= 10$, it follows that $r_{I}(\f) \leq 1.$ Hence $r_I(\f)=1.$ Note that $\mathcal{R}(\f, T)$ is Cohen-Macaulay which implies that $r_{J}(\f) \leq 3-1=2$ by \cite[Theorem 2.3]{viet}. This illustrates that we are getting a better bound (in fact exact bound) by our result. 
\end{example}

The following example shows that the depth assumptions in Theorem \ref{ESf} cannot be
dropped. 
\begin{example}
	Let $R=\C[[X,Y,Z]]/(X^4+Y^4+Z^2) = \C[[x,y,z]]$, where $x,y$ and $z$ denote the image of $X,Y$ and $Z$ respectively in $R.$ Then $R$ is a $2$-dimensional Cohen-Macaulay local ring. Put $\m=(x,y,z).$ We first show that $R$ is normal. Set $T=\C[X,Y,Z]/(f)$, where $f=X^4+Y^4+Z^2.$ Put $l_1=X+Y, l_2=X+iY, l_3=X-Y, l_4=X-iY$. Then $f = X^4+Y^4+Z^2 = Z^2 + l_1l_2l_3l_4 \in \C[X,Y][Z]$. By {\it Eisenstein's criterion}, $f$ is irreducible in $\C[X,Y,Z]$ and hence $T$ is a domain. Using \cite[Theorem 4.4.9]{swansonHuneke}, it follows that $\sing(T)=\V(f,\jac(f))=\V(4x^3, 4y^3,2z)=(0,0,0)$. Thus $T_{\mathfrak{p}}$ is regular for any $\mathfrak{p} \in \Spec(T)$ such that $\mathfrak{p} \neq (x,y,z)$. Moreover, as $T$ is Cohen-Macaulay it satisfies $R_1$ and $S_2$. Therefore by \cite[Theorem 23.8]{matsumuraCRT} it is normal. Now $T \subseteq T_{(x,y,z)} \subseteq \Frac(T)$. So $T_{(x,y,z)}$ is normal. Hence by \cite[Section 32]{matsumuraCRT} we get $R$ is normal. 
	
	Let $\f=\{\overline{\m^n}\}_{n \geq 1}$ and $I=(x,y)R.$ By \cite[Theorem 3.1]{watanabe} we have $\overline{\m^n}= I^n+(z)I^{n-2}$ for all $n \ge 1,$ $r(\f)=2$ and $\gr_{\f}(R)$ is Cohen-Macaulay. Using Proposition \ref{cort}, it follows that $F(\f)$ is Cohen-Macaulay if and only if $I \cap \m \overline{\m^2} = I \m^2.$ Observe that $xz \in I \cap \m \overline{\m^2}$ but if $xz \in I \m^2 = (x^3,x^2y,x^2z,xy^2,xyz,y^3,y^2z),$ then $XZ \in (X^3,X^2Y,X^2Z,XY^2,XYZ,Y^3,Y^2Z,Z^2),$ a contradiction. Therefore, $I \cap \m \overline{\m^2} \neq I \m^2.$ and hence $F(\f)$ is not Cohen-Macaulay. Note that $\mu(\overline{\m^2})=4 < \binom{2+2}{2}=6$ whereas $r(\f) =2.$ It shows that without the assumption of the Cohen-Macaulay property of $F(\f)$, Theorem \ref{ESf} may not hold.
\end{example}

\begin{example}
	Let $R=\mathbb{C}[[X,Y]]/(X^4+Y^2) = \mathbb{C}[[x,y]]$ where $x$ and $y$ are images of $X$ and $Y$ in $R$ respectively.  Consider the equimultiple good filtration $\f = \{\overline{\m^n} \}$. We claim that for all $n \ge 2,$ 
	\[\overline{\m^n} = (x^n,x^{n-2}y).\] 
	Write $X^4+Y^2 = (X^2 +iY)(X^2-iY).$ Observe that the linearity of $Y$ implies that $X^2+iY$ and $X^2-iY$ are irreducible in $T=\mathbb{C}[[X,Y]].$ Hence $f_1 = (X^2+iY)$ and $f_2 = (X^2-iY)$ are minimal primes of $(X^4+Y^2)$ in $T.$ Let $n \ge 2.$ Using the property: An element $a \in \overline{\m^n}$ if and only if image of $a$ in $S_1 = T/(f_1)$ is in $\overline{\m^nS_1}$ and image of $a$ in $S_2 = T/(f_2)$ is in $\overline{\m^nS_2}$, it follows that
	\[ \overline{\m^n} = (x^n, x^2+iy) \cap (x^n, x^2-iy). \]
	Therefore, it is sufficient to show that $(x^n, x^2+iy) \cap (x^n, x^2-iy) = (x^n,x^{n-2}y).$ Set $I=(X^n, X^2+iY), J=(X^n, X^2-iY),$ and $L=(X^n,X^{n-2}Y, X^4+Y^2)$. Then we show that $I \cap J=L$ in $T.$ Clearly, $L \subseteq I \cap J.$ Consider the following exact sequence
	\[ 0 \rightarrow \frac{T}{I \cap J} \rightarrow \frac{T}{I} \oplus \frac{T}{J} \rightarrow \frac{T}{I+J} \rightarrow 0. \]
	Since $\ell(T/I) = n = \ell(T/J)$ and $\ell(T/(I+J)) = \ell(T/(X^2,Y)) = 2$, it follows that $\ell(T/I\cap J)=2n-2.$ As $\ell(T/L)=2n-2$, the claim holds. Observe that $\overline{\m^2} \neq (x)\m$ as $y \notin (x)\m$ but $\overline{\m^n} = (x)\overline{\m^{n-1}}$ for all $n \ge 3$. This implies that $r_{(x)}(\f)=2.$ 
	One can partially recover this observation from Theorem \ref{esthm-mul-fil}. Note that $\{1,\overline{y}\}$ forms a basis of $F(\mathcal{F})$ as a $F(\m)=\oplus_{n \ge 0}\m^n/\m^{n+1}$-module, where $\deg y=2.$ Thus $a=\max\{\deg 1, \deg \overline{y} \} = 2.$ Let $n \geq 3.$ As $\mu(\overline{\m^n}) = 2 < \binom{n+1}{1}=n+1$, and as $n \geq a+1$, Theorem \ref{esthm-mul-fil} implies that $\overline{\m^n} = (x)\overline{\m^{n-1}}$, for all $n \geq 3.$ One can also check that $\mu(\overline{\m^2}) = 2 < (2+1)=3.$ But as $2 \ngeq a+1$, we cannot use Theorem \ref{esthm-mul-fil} in this case.
	
	We now check if Theorem \ref{ESf} can be used to predict the reduction number of the filtration $\mathcal{F}.$ Since $(x) \cap \overline{\m^n} = (x)\overline{\m^{n-1}}$ for all $n \ge 1$, using \cite[Proposition 3.5]{huckabaMarley}, it follows that $\gr_{\f}(R)$ is Cohen-Macaulay. As $\mu(\overline{\m^2}) = 2 < (2+1)=3$, one can conclude that $\overline{\m^2}=x\m$ if the fiber cone of the filtration $\f$ is Cohen-Macaulay. But this fails to be true as $y \notin x\m$. This happens due to non-Cohen Macauleyness of $F(\f).$ Using Proposition \ref{cort}, $F(\f)$ is Cohen-Macaulay if and only if $(x) \cap \m\overline{\m^n} = (x)\m\overline{\m^{n-1}}$ for $n=1,2.$ We claim that 
	\[ xy \in (x) \cap \m\overline{\m^2} \setminus (x)\m^2. \]
	Clearly, $xy \in (x) \cap \m\overline{\m^2}.$ If $xy \in (x)\m^2 = (x^3,x^2y,xy^2)$, then $XY \in (X^3,X^2Y,XY^2,X^4+Y^2),$ a contradiction. Hence $F(\mathcal{F})$ is not Cohen-Macaulay.
\end{example}

\bibliographystyle{plain}
\bibliography{EST}

\begin{thebibliography}{10}

\bibitem{atiyahMacd}
M.~F. Atiyah and I.~G. Macdonald.
\newblock {\em Introduction to commutative algebra}.
\newblock Addison-Wesley Publishing Co., Reading, Mass.-London-Don Mills, Ont.,
  1969.

\bibitem{caviglia}
Giulio Caviglia.
\newblock A theorem of {E}akin and {S}athaye and {G}reen's hyperplane
  restriction theorem.
\newblock In {\em Commutative algebra}, volume 244 of {\em Lect. Notes Pure
  Appl. Math.}, pages 1--5. Chapman \& Hall/CRC, Boca Raton, FL, 2006.

\bibitem{cortadellasZar}
Teresa Cortadellas and Santiago Zarzuela.
\newblock On the depth of the fiber cone of filtrations.
\newblock {\em J. Algebra}, 198(2):428--445, 1997.

\bibitem{ES1976}
Paul Eakin and Avinash Sathaye.
\newblock Prestable ideals.
\newblock {\em J. Algebra}, 41(2):439--454, 1976.

\bibitem{GMV}
Kriti Goel, Vivek Mukundan, and J.~K. Verma.
\newblock Tight closure of powers of ideals and tight {H}ilbert polynomials.
\newblock {\em (submitted)}, 2017.

\bibitem{GS1979}
Shiro Goto and Yasuhiro Shimoda.
\newblock On the {R}ees algebras of {C}ohen-{M}acaulay local rings.
\newblock In {\em Commutative algebra ({F}airfax, {V}a., 1979)}, volume~68 of
  {\em Lecture Notes in Pure and Appl. Math.}, pages 201--231. Dekker, New
  York, 1982.

\bibitem{herzog2017}
J{\"u}rgen Herzog, Maryam~Mohammadei Saem, and Naser Zamani.
\newblock On the number of generators of powers of an ideal.
\newblock {\em arXiv preprint arXiv:1707.07302}, 2017.

\bibitem{hoaZarzuela}
L\^{e}~Tu\^{a}n Hoa and Santiago Zarzuela.
\newblock Reduction number and {$a$}-invariant of good filtrations.
\newblock {\em Comm. Algebra}, 22(14):5635--5656, 1994.

\bibitem{hochster}
M.~Hochster.
\newblock Rings of invariants of tori, {C}ohen-{M}acaulay rings generated by
  monomials, and polytopes.
\newblock {\em Ann. of Math. (2)}, 96:318--337, 1972.

\bibitem{huckabaMarley}
Sam Huckaba and Thomas Marley.
\newblock Hilbert coefficients and the depths of associated graded rings.
\newblock {\em J. London Math. Soc. (2)}, 56(1):64--76, 1997.

\bibitem{hunekeSally}
Craig Huneke and Judith~D. Sally.
\newblock Birational extensions in dimension two and integrally closed ideals.
\newblock {\em J. Algebra}, 115(2):481--500, 1988.

\bibitem{swansonHuneke}
Craig Huneke and Irena Swanson.
\newblock {\em Integral closure of ideals, rings, and modules}, volume 336 of
  {\em London Mathematical Society Lecture Note Series}.
\newblock Cambridge University Press, Cambridge, 2006.

\bibitem{lipman}
Joseph Lipman.
\newblock On complete ideals in regular local rings.
\newblock In {\em Algebraic geometry and commutative algebra, {V}ol.\ {I}},
  pages 203--231. Kinokuniya, Tokyo, 1987.

\bibitem{lyubeznik1986}
Gennady Lyubeznik.
\newblock A property of ideals in polynomial rings.
\newblock {\em Proc. Amer. Math. Soc.}, 98(3):399--400, 1986.

\bibitem{matsumuraCRT}
Hideyuki Matsumura.
\newblock {\em Commutative ring theory}, volume~8 of {\em Cambridge Studies in
  Advanced Mathematics}.
\newblock Cambridge University Press, Cambridge, second edition, 1989.
\newblock Translated from the Japanese by M. Reid.

\bibitem{NR1954}
D.~G. Northcott and D.~Rees.
\newblock Reductions of ideals in local rings.
\newblock {\em Proc. Cambridge Philos. Soc.}, 50:145--158, 1954.

\bibitem{carroll}
Liam O'Carroll.
\newblock Around the {E}akin-{S}athaye theorem.
\newblock {\em J. Algebra}, 291(1):259--268, 2005.

\bibitem{watanabe}
Tomohiro Okuma, Kei-ichi Watanabe, and Ken-ichi Yoshida.
\newblock Normal reduction numbers for normal surface singularities with
  application to elliptic singularities of {B}rieskorn type.
\newblock {\em arXiv preprint arXiv:1804.03795}, 2017.

\bibitem{rees1961}
D.~Rees.
\newblock $\mathfrak a $-transforms of local rings and a theorem on
  multiplicities of ideals.
\newblock {\em Proc. Cambridge Philos. Soc.}, 57:8--17, 1961.

\bibitem{reesAU}
D.~Rees.
\newblock A note on analytically unramified local rings.
\newblock {\em J. London Math. Soc.}, 36:24--28, 1961.

\bibitem{rees}
D.~Rees.
\newblock Hilbert functions and pseudorational local rings of dimension two.
\newblock {\em J. London Math. Soc. (2)}, 24(3):467--479, 1981.

\bibitem{rees1984}
D.~Rees.
\newblock Generalizations of reductions and mixed multiplicities.
\newblock {\em J. London Math. Soc. (2)}, 29(3):397--414, 1984.

\bibitem{sally1975}
Judith~D. Sally.
\newblock On the number of generators of powers of an ideal.
\newblock {\em Proc. Amer. Math. Soc.}, 53(1):24--26, 1975.

\bibitem{teissier1973}
Bernard Teissier.
\newblock Cycles \'evanescents, sections planes et conditions de {W}hitney.
\newblock pages 285--362. Ast\'erisque, Nos. 7 et 8, 1973.

\bibitem{trung}
Ng\^o~Vi\^et Trung.
\newblock Constructive characterization of the reduction numbers.
\newblock {\em Compositio Math.}, 137(1):99--113, 2003.

\bibitem{viet}
Duong~Qu\^oc Vi\^et.
\newblock A note on the {C}ohen-{M}acaulayness of {R}ees algebras of
  filtrations.
\newblock {\em Comm. Algebra}, 21(1):221--229, 1993.

\bibitem{zariski}
Oscar Zariski.
\newblock Polynomial {I}deals {D}efined by {I}nfinitely {N}ear {B}ase {P}oints.
\newblock {\em Amer. J. Math.}, 60(1):151--204, 1938.

\bibitem{zariskiSamuel}
Oscar Zariski and Pierre Samuel.
\newblock {\em Commutative algebra. {V}ol. {II}}.
\newblock The University Series in Higher Mathematics. D. Van Nostrand Co.,
  Inc., Princeton, N. J.-Toronto-London-New York, 1960.

\end{thebibliography}
\end{document}